\documentclass[review,a4paper]{elsarticle}
\usepackage[margin=1in]{geometry}
\usepackage{lineno,hyperref}
\usepackage{dsfont}
\usepackage{tikz}
\usepackage{amsmath,amssymb,amsfonts,amsthm}

\usepackage{algorithm}
\usepackage{bm}
\usepackage{bookmark}
\usepackage{xfrac}
\newtheorem{thm}{Theorem}[section]
\newtheorem{lem}[thm]{Lemma}
\newtheorem{proposition}[thm]{Proposition}
\newtheorem{cor}[thm]{Corollary}
\theoremstyle{definition}
\newtheorem{dfn}[thm]{Definition}
\newdefinition{con}[thm]{Construction}
\newdefinition{rmk}[thm]{Remark}

\theoremstyle{remark}
\numberwithin{equation}{section}

\journal{Journal of \LaTeX\ Templates}

\makeatletter
\def\ps@pprintTitle{%
   \let\@oddhead\@empty
   \let\@evenhead\@empty
   \def\@oddfoot{\reset@font\hfil\thepage\hfil}
   \let\@evenfoot\@oddfoot
}
\makeatother

\begin{document}

\begin{frontmatter}

\title{Directed Strongly Regular Cayley Graphs on Dihedral groups $D_n$}
\author[abc]{Yiqin He}
\ead{2014750113@smail.xtu.edu.cn}
\author[rvt]{Bicheng Zhang\corref{cor1}}
\ead{zhangbicheng@xtu.edu.cn}
\author[abc]{Rongquan Feng}
\ead{fengrq@math.pku.edu.cn}
\address[abc]{School of Mathematical Sciences, Peking University, Beijing, 100871, PR China}
\address[rvt]{School of Mathematics and Computational Science, Xiangtan Univerisity, Xiangtan, Hunan, 411105, PR China}
\cortext[cor1]{Corresponding author}



%

\begin{abstract}
In this paper,\;we characterize some certain directed strongly regular Cayley graphs on Dihedral groups $D_{n}$,\;where $n\geqslant 3$ is a positive integer.\;
\end{abstract}
\end{frontmatter}
\section{Introduction}
\subsection{Overview}
A \emph{directed strongly regular graph} (DSRG) with parameters $( n, k, \mu ,\lambda , t)$ is a $k$-regular directed graph on $n$ vertices such that every vertex is incident with $t$ undirected edges,\;and the number of paths of length two from a vertex $x$ to a vertex $y$ is $\lambda$ if there is an edge directed
from $x$ to $y$ and it is $\mu$ otherwise.\;A DSRG with $t=k$ is an \emph{(undirected) strongly regular graph} (SRG).\;Duval showed that DSRGs
with $t=0$ are the \emph{doubly regular tournaments}.\;It is therefore usually assumed that $0<t<k$.\;The DSRGs which satisfy the condition $0<t<k$ are called genuine DSRGs.\;The DSRGs appear on this paper are all genuine.

Let $D$ be a directed graph with $n$ vertices.\;Let $A=\mathbf{A}(D)$ denote the adjacency matrix of $D$,\;and let $I = I_n$ and
$J = J_n$ denote the $n\times n$ identity matrix and all-ones matrix,\;respectively.\;Then $D$ is a directed strongly regular graph
with parameters $(n,k,\mu,\lambda,t)$ if and only if (i) $JA = AJ = kJ$ and (ii) $A^2=tI+\lambda A+\mu(J-I-A)$.\;

Let $G$ be a finite (multiplicative) group and $S$ be a subset of  $G\backslash\{e\}$.\;The Cayley graph of $G$ generated
by $S$,\;denoted by $\mathbf{Cay}(G,S)$,\; is the digraph $\Gamma$ such that $V(\Gamma)=G$ and $g\rightarrow h$ if and only if
$g^{-1}h\in S$,\;for any $g,\;h\in G$.

Let $C_n=\langle x\rangle$ be a cyclic multiplicative group of order $n$.\;The \emph{dihedral group} $D_n$ is the group of symmetries of a
regular polygon,\;and it can be viewed as a semidirect product of two cyclic groups $C_n=\langle x\rangle$ of order $n$
and $C_2=\langle \tau\rangle$ of order $2$.\;The presentation of $D_n$
is $D_n=C_n\rtimes C_2=\langle x,\tau|x^n=1,\tau^2=1,\tau x=x^{-1}\tau\rangle$.\;
The cyclic group $C_n$ is a normal subgroup of $D_n$ of index $2$.\;
Note that each subset $S$ of the dihedral group $D_n$ can be written by the form $S=X\cup Y\tau$ for
some subsets $X,Y\subseteq C_n$,\;we denote the Cayley graph $\mathbf{Cay}(D_n, X\cup Y\tau)$ by $Dih(n,X,Y)$.\;

In this paper,\;we focus on the directed strongly regular Cayley graphs on dihedral groups.\;The \emph{Cayley graphs} on dihedral groups are called \emph{dihedrants}.\;A dihedrant which is a DSRG is called \emph{directed strongly regular dihedrant}.\;

We characterize all the directed strongly regular dihedrants $Dih(n,X,X)$ for any integer $n\geqslant3$.\;Note that the dihedrants $Dih(n,X,X)$ and $Dih(n,\sigma(X),x^b\sigma(X))$ are isomorphic for any $\sigma\in \mathbf{Aut}(C_n)$ and $b\in \mathbb{Z}_n$.\;Hence,\;we actually characterize all the  directed strongly regular dihedrants $Dih(n,\sigma(X),x^b\sigma(X))$ for any $\sigma\in \mathbf{Aut}(C_n)$ and $b\in \mathbb{Z}_n$.\;

\subsection{Notation and terminology}

Let $G$ denote a finite group with identity $e$,\;and let $X$ denote a nonempty subset of $G$.\;We denote by $X^{(-1)}$ the set $\{x^{-1}:x\in X\}$.\;

The following notataion will be used.\;Let $A$ be a multiset together with a \emph{multiplicity function} $\Delta_{A}$,\;where $\Delta_{A}(a)$ counting how many times $a$ occurs in the multiset $A$.\;We say $a$ belongs to $A$\;(i.e.\;$a\in A$)\;if $\Delta_A(a)>0$.\;
In the following,\;$A$ and $B$ are multisets,\;with multiplicity functions $\Delta_{A}$ and $\Delta_{B}$.\;

\begin{itemize}
  \item $|A|=\sum\limits_{a\in A}\Delta_A(a)$;
  \item $A\subseteq B$ if $\Delta_A(a)\leqslant\Delta_B(a)$ for any $a\in B$;
  \item \textbf{Union,\;$A\uplus B$}:\;the union of multisets $A$ and $B$,\;is defined by $\Delta_{A\uplus B}=\Delta_A+\Delta_B$;
  \item \textbf{Scalar multiplication,\;$n\oplus A$}:\;the scalar multiplication of a multiset $A$ by a natural number $n$,\;is defined by $\Delta_{n\oplus A}=n\Delta_A$.
  \item \textbf{Difference,\;$A\setminus B$}:\;the difference of multisets $A$ and $B$,\;is defined by $\Delta_{A\setminus B}(x)=\max\{\Delta_A(a)-\Delta_B(a),0\}$ for any $a\in A$.\;
\end{itemize}
If $A$ and $B$ are usual sets,\;we use $A\cup B$,\;$A\cap B$ and  $A\setminus B$ denote the usual union,\;intersection and difference of $A$ and $B$.\;For example,\;if $A=\{1,2\}$ and $B=\{1,3\}$,\;then $A\uplus B=\{1,1,2,3\}$,\;$A\cup B=\{1,2,3\}$,\;$2\oplus A=\{1,1,2,2\}$,\;$A\setminus B=\{2\}$ and $\{1,1,2,2\}\setminus B=\{1,2,2\}$.\;

Let $C_n=\langle x\rangle$ be a cyclic multiplicative group of order $n$.\;Let $v$ be a divisor of $n$,\;then $\langle x^{v}\rangle$ is a normal subgroup of $C_n=\langle x\rangle$ and hence $C_n=\bigcup_{j=0}^{v-1}x^j\langle x^{v}\rangle$ is the left coset decomposition of $C_n$ with respect to the subgroup $\langle x^{v}\rangle$.\;Let $T$ be a multisubet of $\{e,x^1,\cdots,x^{v-1}\}$,\;then the notation $T\langle x^{v}\rangle$ means that
\[T\langle x^{v}\rangle=\biguplus_{t\in T}(\Delta_T(t)\oplus t\langle x^{v}\rangle).\]
For example,\;if $T=\{e,e,x\}$,\;then $T\langle x^{v}\rangle=\langle x^{v}\rangle\uplus\langle x^{v}\rangle\uplus(x\langle x^{v}\rangle)$.\;

Throughout this paper,\;let $\mathbb{Z}_n=\{0,1,2,\cdots,n-1\}$ be the modulo $n$ residue class ring.\;For
a positive divisor $v$ of $n$ let $v\mathbb{Z}_n=\{0,v,2v,\cdots,n-v\}$ be the subgroup of the additive group of $\mathbb{Z}_n$ of order $\frac{n}{v}$.\;
For a multisubset $A$ of $\mathbb{Z}_n$,\;let $$x^A=\biguplus_{i\in A}\Delta_A(i)\oplus\{x^i\},\;$$
then $x^A$ is a multisubset of $C_n$.\;Observe that $C_n=x^{\mathbb{Z}_n}$ and $\langle x^{v}\rangle=x^{v\mathbb{Z}_n}$ for any $v|n$.\;Let $A,B$ be the mutlisubsets of $\mathbb{Z}_n$,\;let $-A=\{-a|a\in A\}$ and $A+B=\{a+b|a\in A,b\in B\}$,\;where the elements in the $-A$ and $A+B$ are counted multiplicities.\;

For example,\;let $n=5$,\;$A=\{1,1,3\}$ and $B=\{2,4\}$,\;then $x^A=\{x,x,x^3\}$,\;$-A=\{4,4,2\}$ and $A+B=\{1+2,1+4,1+2,1+4,3+2,3+4\}=\{3,0,3,0,0,2\}$.\;

\subsection{The main theorem}
In this paper,\;we focus on the directed strongly regular dihedrants.\;We now give some known directed strongly regular dihedrants.
\begin{thm}(\cite{K2})\label{t-1.1}Let $n$ be odd and let $X,Y\subset C_n$ satisfy the following conditions:\\
(i)\;$\overline{X}+\overline{X^{(-1)}}=\overline{C_n}-e$,\\
(ii)\;$\overline{Y}\;\overline{Y^{(-1)}}-\overline{X}\;\overline{X^{(-1)}}=\varepsilon\overline{C_n}$,$\varepsilon\in\{0,1\}$.\\
Then $\mathbf{Cay}(D_n,X\cup aY)$ is a DSRG with parameters $(2n,n-1+\varepsilon,\frac{n-1}{2}+\varepsilon,\frac{n-3}{2}+\varepsilon,\frac{n-1}{2}+\varepsilon)$.\;In particular,\;if $X$ satisfies $(i)$ and $Y=Xg$ or $X^{(-1)}g$ for some $g\in C_n$,\;then $\mathbf{Cay}(D_n,X\cup aY)$ is a DSRG with parameters $(2n,n-1,\frac{n-1}{2},\frac{n-3}{2},\frac{n-1}{2})$.\;
\end{thm}
We can say more when $n$ is a odd prime.
\begin{thm}(\cite{K2})\label{t-1.2}Let $n$ be odd prime and let $X,Y\subset C_n$ and $b\in D_n\setminus C_n$,\;Then the Cayley graph $\mathbf{Cay}(D_n,X\cup bY)$ is a DSRG if and only $X,Y$ satisfy the conditions of Theorem \ref{t-1.1}.\;
\end{thm}

\begin{thm}(\cite{FI})\label{t-1.3}Let $n$ be even,\;$c\in C_n$ be an involution and let $X,Y\subset C_n$ such that:\\
(i)\;$\overline{X}+\overline{X^{(-1)}}=\overline{C_n}-e-c$,\\
(ii)\;$\overline{Y}=\overline{X}$ or $\overline{Y}=\overline{X^{(-1)}}$,\\
(ii)\;$\overline{Xc}=\overline{X^{(-1)}}$.\;\\
Let $b\in D_n\setminus C_n$,\;then the Cayley graph $\mathbf{Cay}(D_n,X\cup bY)$
is a DSRG with parameters $(2n,n-1,\frac{n}{2}-1,\frac{n}{2}-1,\frac{n}{2})$.\;
\end{thm}

For an odd prime $p$,\;the characterization of directed strongly regular dihedrant $Dih(p,X,Y)$ has been achieved in \cite{K2},\;which was presented in Theorem \ref{t-1.2}.\;

In this paper,\;we characterize all the directed strongly regular dihedrants $Dih(n,X,X)$.\;We obtain the following theorem.\;
\begin{thm}\label{t-1.4}The dihedrant $Dih(n,X,X)$ is a DSRG with parameters $(2n,2|X|,\mu,\lambda,t)$ if and only if one of the following  holds:\\
$(a)$\;Let $v=\frac{n}{\mu-\lambda}$.\;There exists a subset $T$ of $\{x,x^2,\cdots,x^{v-1}\}$ satisfies\\
\hspace*{15pt}$(i)$\;$X=T\langle x^v\rangle$;\\
\hspace*{15pt}$(ii)$\;$X\cup X^{(-1)}=C_n\setminus\langle x^v\rangle$.\\
$(b)$\;Let $v=\frac{n}{\mu-\lambda}$.\;There exists a subset $T'$ of $ \{x,x^2,\cdots,x^{2v-1}\}$ satisfies\\
\hspace*{15pt}$(i)$\;$X=T'\langle x^{2v}\rangle$;\\
\hspace*{15pt}$(ii)$\;$X\uplus X^{(-1)}=(C_n\setminus\langle x^{2v}\rangle)\uplus x^v\langle x^{2v}\rangle$;\\
\hspace*{15pt}$(iii)$\;$X\cup x^vX=C_n$.
\end{thm}

\begin{rmk}The Lemma 2.4 in \cite{MI} asserts that the dihedrants $Dih(n,\sigma(X),x^b\sigma(X))$ and $Dih(n,X,X)$ are isomorphic for any $\sigma\in \mathbf{Aut}(C_n)$ and $b\in \mathbb{Z}_n$.\;Hence,\;we actually characterize all the  directed strongly regular dihedrants $Dih(n,\sigma(X),x^b\sigma(X))$ for any $\sigma\in \mathbf{Aut}(C_n)$ and $b\in \mathbb{Z}_n$.\;
\end{rmk}
\section{Preliminary}

\subsection{Properties of DSRG}
Duval \cite{A} developed necessary conditions on the parameters of $(n,k,\mu,\lambda,t)$-DSRG and calculated the spectrum of a DSRG.\;
\begin{proposition}(see \cite{A})\label{p-2.1}
A DSRG with parameters $( n,k,\mu ,\lambda ,t)$ with $0<t<k$ satisfy
\begin{equation}\label{2.1}
k(k+(\mu-\lambda ))=t+\left(n-1\right)\mu,
\end{equation}
\[{{d}^{2}}={{\left( \mu -\lambda  \right)}^{2}}+4\left(t-\mu\right)\text{,}d|2k-(\lambda-\mu)(n-1),\]
\[\frac{2k-(\lambda-\mu)(n-1)}{d}\equiv n-1(mod\hspace{2pt}2)\text{,}\left|\frac{2k-(\lambda-\mu)(n-1)}{d}\right|\leqslant n-1,\]
where d is a positive integer, and
$$ 0\leqslant \lambda <t<k,0<\mu\leqslant t<k,-2\left( k-t-1 \right)\leqslant \mu -\lambda \leqslant 2\left( k-t \right).$$
\end{proposition}
\begin{rmk}If $0<t=\mu<k$,\;then $\lambda-\mu<0$.\;
\end{rmk}
\begin{proposition}(see \cite{A})\label{p-2.2}
A DSRG with parameters $( n, k, \mu ,\lambda , t)$ has three distinct integer eigenvalues
$$ k>\rho =\frac{1}{2}\left( -\left( \mu -\lambda  \right)+d \right)>\sigma =\frac{1}{2}\left( -\left( \mu -\lambda  \right)-d \right),\; $$
The multiplicities are
$$  1,\;m_\rho=-\frac{k+\sigma \left( n-1 \right)}{\rho -\sigma }\text{\;and\;}m_\sigma=\frac{k+\rho \left( n-1 \right)}{\rho -\sigma },\;$$
respectively.\;
\end{proposition}
\begin{proposition}(see \cite{A})\label{p-2.3}
If $G$ is a DSRG with parameters $( n, k, \mu ,\lambda , t)$,\;then the complementary $G'$ is also a DSRG with parameters $(n',k',\mu',\;\lambda' ,t')$,\;where $k'=(n-2k)+(k-1)$,\;$\lambda'=(n-2k)+(\mu-2)$,\;$t'=(n-2k)+(t-1)$,\;$\mu'=(n-2k)+\lambda$.\;
\end{proposition}

\begin{dfn}\label{d-2}(Group Ring) For any finite (multiplicative) group $G$ and ring $R$,\;the \emph{Group Ring} $R[G]$ is defined as the set of all formal sums of elements of $G$,\;with coefficients from $R$.\;i.e.,
\[R[G]=\left\{\sum_{g\in G}r_gg|r_g\in R,\; r_g\neq 0\text{ for finite g}\right\}.\;\]
The operations $+$ and $\cdot$ on $R[G]$ are given
by
\[\sum_{g\in G}r_gg+\sum_{g\in G}s_gg=\sum_{g\in G}(r_g+s_g)g,\;\]
\[\left(\sum_{g\in G}r_gg\right)\cdot\left(\sum_{g\in G}s_gg\right)=\left(\sum_{g\in G}t_gg\right),\;t_g=\sum_{g'g''=g}r_{g'}s_{g''}.\]
\end{dfn}

For any multisubset $X$ of $G$,\;Let $\overline{X}$ denote the element of the group ring $R[G]$ that is the sum of all elements of $X$.\;i.e., \[\overline{X}=\sum_{x\in X}\Delta_X(x)x.\]
In particular,\;if $X$ is a subset of $G$,\;then
\[\overline{X}=\sum_{x\in X}x.\]
The Lemma below allows us to express a sufficient and necessary condition for a Cayley graph to be  strongly regular in terms of group ring $\mathbb{Z}[G]$.
\begin{lem}\label{l-2.6}
The Cayley graph $\mathbf{Cay}(G, S)$ of $G$ with respect to $S$ is a DSRG with parameters $(n, k, \mu,  \lambda, t)$ if and only if $|G| = n$, $|S|= k$, and
\[\overline{S}^2=te+\lambda\overline{S}+\mu(\overline{G}-e-\overline{S}).\]
\end{lem}

\subsection{Character theory of abelian groups}
A \emph{character} $\chi$ of a finite abelian group is a homomorphism from $G$ to $\mathbb{C}^*$,\;the multiplicative group of $\mathbb{C}$.\;All characters of $G$ form a group under the multiplication $\chi\chi'(g)=\chi(g)\chi'(g)$ for any $g\in G$,\;which is denoted by $\widehat{G}$ and it is called the \emph{character group} of ${G}$.\;It is easy to see that $\widehat{G}$ is isomorphic to $G$.\;Every character $\chi\in\widehat{G}$ of $G$ can be extended to a homomorphism from $\mathbb{C}[G]$ to $\mathbb{C}$  by\;
\[\chi\left(\sum_{g\in G}a_gg\right)=\sum_{g\in G}a_g\chi(g).\]

We also have the following Lemma.\;
\begin{lem}\label{l-2.7}(Inversion formula)Let $G$ be an abelian group.\;Let $A,B\in \mathbb{C}[G]$,\;then $A=B$ if and only if
$\chi(A)=\chi(B)$  for any $\chi\in \widehat{G}$,\;
\end{lem}

Let $\zeta_n$ be a fixed primitive $n$-th root of unity,\;then
$\widehat{C_n}=\{\chi_j|j\in \mathbb{Z}_n\}$,\;where $\chi_j(x^i)=\zeta_n^{ij}$ for $0\leqslant i,j\leqslant n-1$.\;The character $\chi_0$ is called  the principal character of $\widehat{C_n}$ and the characters $\chi_1,\chi_2,\cdots,\chi_{n-1}$ are  called the nonprincipal characters of $\widehat{C_n}$.\;Let $v$ be a divisor of $n$,\;and let $\pi_{v}$ be the natural projection from $C_n=\langle x\rangle$ to the quotient group $\langle x\rangle/\langle x^{v}\rangle=\langle \pi_{v}(x)\rangle$,\;this quotient group is a cyclic group of order $\frac{n}{v}$ generated by $\pi_{v}(x)$ and  $\langle x\rangle/\langle x^{v}\rangle=\{\pi_{v}(e),\pi_{v}(x),\cdots,\pi_{v}(x^{v-1})\}$.\;

\begin{lem}\label{l-2.8}There is a bijective correspondence between the set of characters of $\langle x\rangle/\langle x^{v}\rangle=\langle \pi_{v}(x)\rangle$ and the set of characters of $C_n$ such that $\langle x^{v}\rangle\subseteq \ker\chi$.\;i.e.,\;
\[\widehat{\langle x\rangle/\langle x^{v}\rangle}\leftrightarrow\{\chi\in \widehat{C_n}|\langle x^{v}\rangle\subseteq\ker\chi\}=\{\chi_{\frac{n}{v}j}\in \widehat{C_n}|0\leqslant j\leqslant v-1\}.\]
\end{lem}
\begin{proof}Let $\chi\in \widehat{C_n}$ be a character of $C_n$ such that $\langle x^{v}\rangle\subseteq \ker\chi$.\;Then $\chi$ induces a character of
$\langle x\rangle/\langle x^{v}\rangle$ in the following way.\;Let
\[\chi'(\pi_{v}(x^i)\rangle)=\chi(x^i)\]
for $0\leqslant i\leqslant v-1$.\;This definition is well-defined and it is clear that $\chi'$ is a character of $\langle x\rangle/\langle x^{v}\rangle$.\;Conversely,\;if $\chi'$ is a character of $\langle x\rangle/\langle x^{v}\rangle$,\;then we let
\[\chi(x^i)=\chi'(\pi_{v}(x^i))\]
for $0\leqslant i\leqslant n-1$,\;then $\chi$ becomes a character of $C_n$ with $\langle x^{v}\rangle\subseteq \ker\chi$.\;This completes the proof.\;
\end{proof}
\begin{rmk}Furthermore,\;if $\widehat{\langle x\rangle/\langle x^{v}\rangle}=\{\chi_0',\chi_1',\cdots,\chi_{v-1}'\}$,\;then $\chi_j'\in \widehat{\langle x\rangle/\langle x^{v}\rangle}$ is the character induced by the character $\chi_{\frac{n}{v}j}\in \widehat{C_n}$,\;for $0\leqslant j\leqslant v-1$.\;
\end{rmk}
\subsection{Fourier Transformation}
Throughout this subsection $n$ will denote a fixed positive integer,\;$\mathbb{Z}_n=\{0,1,\cdots,n-1\}$ is the modulo $n$ residue class ring.\;Recall for a positive divisor $v$ of $n$,\;let $v\mathbb{Z}_n=\{0,v,2v,\cdots,n-v\}$ be the subgroup of the additive group of $\mathbb{Z}_n$.\;

The following statement and  notations  are coincided with \cite{MI}.\;Let $\mathbb{Z}_n^{\ast}$  be the multiplicative group
of units in the ring $\mathbb{Z}_n$.\;Then $\mathbb{Z}_n^{\ast}$ has an action on $\mathbb{Z}_n$ by multiplication and hence $\mathbb{Z}_n$ is a union of some  $\mathbb{Z}_n^{\ast}$-orbits.\;Each $\mathbb{Z}_n^{\ast}$-orbit consists of all elemnents of a given order in the additive group $\mathbb{Z}_n$.\;If $v$ is a positive divisor of $n$,\;we denote the $\mathbb{Z}_n^{\ast}$-orbit containing all elements of order $v$ with $\mathcal{O}_v$.\;Thus
\[\mathcal{O}_v=\left\{z\bigg|z\in \mathbb{Z}_n,\;\frac{n}{(n,z)}=v\right\}=\left\{c{\frac{n}{v}}\bigg|1\leqslant c\leqslant v,(v,c)=1\right\}\]
and $|\mathcal{O}_v|=\varphi(v)$.

Let $\zeta_n$ be a fixed primitive $n$-th root of unity and $\mathbb{F}=\mathbb{Q}(\zeta_n)$ the $n$-th
\emph{cyclotomic field} over $\mathbb{Q}$.\;Further,\;let $\mathbb{F}^{\mathbb{Z}_n}$ be the $\mathbb{F}$-vector space of all
functions $f:\mathbb{Z}_n\rightarrow\mathbb{F}$ mapping from the residue class ring $\mathbb{Z}_n$ to the field $\mathbb{F}$ (with the
scalar multiplication and addition defined point-wise).\;The $\mathbb{F}$-algebra obtained from $\mathbb{F}^{\mathbb{Z}_n}$ by
defining the multiplication point-wise will be denoted by $(\mathbb{F}^{\mathbb{Z}_n},\cdot)$.\;The $\mathbb{F}$-algebra obtained from $\mathbb{F}^{\mathbb{Z}_n}$ by defining the multiplication as \emph{convolution} will be denoted by $(\mathbb{F}^{\mathbb{Z}_n},\ast)$,\;where the convolution is defined by:\;$(f\ast g)(z)=\sum_{i\in \mathbb{Z}_n}f(i)g(z-i)$.\;The \emph{Fourier transformation} which is an isomorphsim of $\mathbb{F}$-algebra $(\mathbb{F}^{\mathbb{Z}_n},\ast)$ and $(\mathbb{F}^{\mathbb{Z}_n},\cdot)$ is defined by
\begin{equation}\label{2.2}
\mathcal{F}:(\mathbb{F}^{\mathbb{Z}_n},\ast)\rightarrow (\mathbb{F}^{\mathbb{Z}_n},\cdot),\;\;\;\;(\mathcal{F}f)(z)=\sum_{i\in \mathbb{Z}_n}f(i)\zeta_n^{iz}.
\end{equation}
then,\;$\mathcal{F}(f\ast g)=\mathcal{F}(f)\mathcal{F}(g)$ for $f,g\in \mathbb{F}^{\mathbb{Z}_n}$.\;It also obeys the \emph{inversion formula}
\begin{equation}\label{2.3}
\mathcal{F}(\mathcal{F}(f))(z)=nf(-z).
\end{equation}
Observe that for any multisubsets $A$ and $B$ of $\mathbb{Z}_n$,\;the following hold.
\begin{equation}\label{2.4}
\mathcal{F}\Delta_{(-A)}=\overline{\mathcal{F}\Delta_{A}},\;\mathcal{F}\Delta_{A+B}=\mathcal{F}(\Delta_A\ast\Delta_B)=(\mathcal{F}\Delta_{A})(\mathcal{F}\Delta_{B}).
\end{equation}
Recall that $\widehat{C_n}=\{\chi_j|j\in \mathbb{Z}_n\}$,\;then
\begin{equation}\label{2.5}
(\mathcal{F}\Delta_A)(z)=\chi_z(\overline{x^A})
\end{equation}
for $z\in \mathbb{Z}_n$.\;The Fourier transformation of characteristic functions of  additive subgroups in $\mathbb{Z}_n$  can be easily computed.\;For a positive divisor $v$ of $n$,\;
\begin{equation}\label{2.6}
\mathcal{F}\Delta_{v\mathbb{Z}_n}=\frac{n}{v}\Delta_{\frac{n}{v}\mathbb{Z}_n},\;\mathcal{F}\Delta_{\mathbb{Z}_n}=n\Delta_{0},\;\text{and\;}\mathcal{F}\Delta_{0}=\Delta_{\mathbb{Z}_n}=1.
\end{equation}
The following lemma will be used in this paper.
\begin{lem}(\cite{MI})\label{l-2.10}Suppose that $f:\mathbb{Z}_n\rightarrow\mathbb{F}$ is a function such that ${\rm{Im}}(f)\subseteq \mathbb{Q}$.\;Then ${\rm{Im}}(\mathcal{F}f)\subseteq \mathbb{Q}$  if and only if $f=\sum_{v|n}\alpha_v\Delta_{\mathcal{O}_v}$ for some $\alpha_v\in\mathbb{Q}$.
\end{lem}
The value of Fourier transformation of the characteristic function of an orbit $\mathcal{O}_v$ also knows as the Ramanujan's sum,\;i.e.,\;
\begin{equation}\label{2.7}
(\mathcal{F}\Delta_{\mathcal{O}_v})(z)=\mu\left(\frac{v}{(v,z)}\right)\frac{\varphi(v)}{\varphi\left(\frac{v}{(v,z)}\right)}\in\mathbb{Z}.
\end{equation}

In this paper,\;we usually just consider the additive group of $\mathbb{Z}_n$.\;

\section{Some lemmas}
In this section,\;we prove some lemmas which will be used to characterize directed strongly regular dihedrants $Dih(n,X,X)$.\;At first,\;we need the following lemma.
\begin{lem}(\cite{Arasu1996Exponent})\label{l-3.1}Let $p$ be a prime and let $G$ be a finite abelian group with a cyclic Sylow $p$-subgroup of order $p^l$ with $l=0$ permitted. Let $P_i$ be the cyclic subgroup of order
$p^i$ for $i=0,1,\cdots,l$. Suppose $Y$ is an element of the group ring  $\mathbb{Z}[G]$ that satisfies
\[\chi(Y)\equiv 0 \;\mathrm{mod}\;p^f\]
 for some positive integer $f$ and all nonprincipal characters $\chi$ of
$G$.\;Moreover,\;if $l<f$ assume that $\chi_0(Y)\equiv 0 \;\mathrm{mod}(p^f)$ for the principal character
$\chi_0$.\;Then $Y$ can be expressed in the form
\[Y=p^f X_0+p^{f-1}\overline{P_1}X_1+\cdots+p^{f-m}\overline{P_m}X_m.\]
where $m=\min\{l,f\}$ and the $X_i$ are elements of $\mathbb{Z}[G]$.\;Furthermore,\;if the coefficients
of $Y$ are nonnegative, then the $X_i$ can be chosen to have nonnegative integer coefficients.
\end{lem}

\subsection{Generalization of Lemma \ref{l-3.1}}
Throughout this section,\;$U$ is a multisubset of $C_n=\langle x\rangle$ such that $\Delta_{U}(u)\leqslant2$ for any $u\in U$($\Delta_{U}\leqslant2$ for short),\;i.e.,\;each element in $U$ occurs at most twice.\;Define $\nu_p(z)$ as the maximum  power of the prime $p$ that divides $n$.\;We use $\nu(n)$ to mean $\nu_2(n)$ when $p=2$.\;

Let $n=2^{\nu(n)}\prod\limits_{i=1}^sp_i^{\alpha_i}$ be the prime factorization of $n$,\;where $p_1,p_2,\cdots,p_s$ are odd primes.\;Let $P_i^{(p)}=\langle x^{\frac{n}{p^i}}\rangle$ be the cyclic subgroup of $C_n$ of order $p^i$,\;for $i=0,1,\cdots,\nu_p(n)$.\;

Let $c$ be a positive divisor of $n$,\;and $c=2^{\nu(c)}\prod\limits_{i=1}^sp_i^{\beta_i}$ be the prime factorization of $c$,\;where $\beta_i\leqslant\alpha_i$ for $1\leqslant i\leqslant s$.\;We have the following lemmas.
\begin{lem}\label{l-3.2}Let $U$ be a multisubset of $C_n$ such that $\Delta_{U}\leqslant2$.\;If
\[\chi(\overline{U})\equiv 0 \;\mathrm{mod}\;c\]
for all nonprincipal characters $\chi$ of $C_n$.\;Then for each odd prime divisor $p_i$ of $n$,\;there exists a multisubset $E_i$ of $\{0,1,\cdots,{\frac{n}{p_i^{\beta_i}}-1}\}$ such that $\Delta_{E_i}\leqslant2$ and
\begin{equation}
\begin{aligned}
U=x^{E_i}\langle x^{\frac{n}{p_i^{\beta_i}}}\rangle.
\end{aligned}
\end{equation}
\end{lem}
\begin{proof}Note that
\[\chi(\overline{U})\equiv 0 \;\mathrm{mod}\;c\;\Rightarrow \chi(\overline{U})\equiv 0 \;\mathrm{mod}\;p_i^{\beta_i}\]
 for all nonprincipal characters $\chi$ of $C_n$,\;then from Lemma  \ref{l-3.1} with $m=\min\{\beta_i,\alpha_i\}=\beta_i$,\;there are elements $X_0,X_1,\cdots,X_{\beta_i}$ with non-negative coefficients in $\mathbb{Z}[C_{n}]$ such that
\[\overline{U}=p_i^{\beta_i} X_0+p_i^{{\beta_i}-1}\overline{P_1^{(p_i)}}X_1+\cdots+p_i^{\beta_i-\beta_i}\overline{P_{\beta_i}^{(p_i)}}X_\beta.\]
It is clear that all the coefficients in $p_i^{\beta_i-j}\overline{P_j^{(p_i)}}X_j$ at least $p_i$ provided $X_j\neq 0$,\;for any $0\leqslant j\leqslant \beta_i-1$.\;This gives that  $X_0=X_1=\cdots=X_{\beta_i-1}=0$ since all the coefficients in $\overline{U}$ don't exceed $2<p_i$.\;Hence \[\overline{U}=X_{\beta_i}\overline{P_{\beta_i}^{(p_i)}},\]
this claims that the multiset $U$ is a union of some cosets of $P_{\beta_i}^{(p_i)}$ in $C_n$ and hence $X_{\beta_i}$ can be chosen as a multisubset of $\{e=x^0,x^1,\cdots,x^{\frac{n}{p_i^{\beta_i}}-1}\}$.\;Let $X_{\beta_i}=\overline{x^{E_i}}$ for some multisubset $E_i$ of $\{0,1,\cdots,{\frac{n}{p_i^{\beta_i}}-1}\}$,\;then $\Delta_{E_i}\leqslant2$ holds clearly and
\[{U}=x^{E_i}{P_{\beta_i}^{(p_i)}}=x^{E_i}\langle x^{\frac{n}{p_i^{\beta_i}}}\rangle.\]
This  completes the proof.\;
\end{proof}

\begin{lem}\label{l-3.3}Let $c$ be an even number and $U$ be a multisubset of $C_n$ such that $\Delta_{U}\leqslant2$.\;If
\[\chi(\overline{U})\equiv 0 \;\mathrm{mod}\;c\]
for all nonprincipal characters $\chi$ of $C_n$.\;Then there exists a multisubset $E_0$ of $\{0,1,\cdots,{\frac{2n}{2^{\nu(c)}}-1}\}$ such that
\begin{equation}
\begin{aligned}
U=x^{E_0}\langle x^{\frac{2n}{2^{\nu(c)}}}\rangle.
\end{aligned}
\end{equation}
\end{lem}
\begin{proof}Note that
\[\chi(\overline{U})\equiv 0 \;\mathrm{mod}\;c\Rightarrow \chi(\overline{U})\equiv 0 \;\mathrm{mod}\;2^{\nu(c)}\]
 for all nonprincipal characters $\chi$ of $C_n$,\;then from Lemma  \ref{l-3.1} with $m=\min\{{\nu(c)},{\nu(n)}\}={\nu(c)}$,\;there are elements $X_0,X_1,\cdots,X_{\beta_i}$ with non-negative coefficients in $\mathbb{Z}[C_{n}]$ such that
\[\overline{U}=2^{\nu(c)} X_0+2^{\nu(c)-1}\overline{P_1^{(2)}}X_1+\cdots+2^{\nu(c)-{\nu(c)}}\overline{P_{\nu(c)}^{(2)}}X_{\nu(c)}.\]
It is clear that all the coefficients in $2^{\nu(c)-j}\overline{P_j^{(2)}}X_j$ at least $4$ provided $X_j\neq 0$,\;for any $0\leqslant j\leqslant \nu(c)-2$.\;This gives that  $X_0=X_1=\cdots=X_{\nu(c)-2}=0$ since all the coefficients in $\overline{U}$ don't exceed $2$.\;Hence
\begin{equation*}
\begin{aligned}
\overline{U}&=2\overline{P_{\nu(c)-1}^{(2)}}X_{\nu(c)-1}+\overline{P_{\nu(c)}^{(2)}}X_{\nu(c)}\\
&=2\overline{\langle x^{\frac{2n}{2^{\nu(c)}}}\rangle}X_{\nu(c)-1}+\overline{\langle x^{\frac{n}{2^{\nu(c)}}}\rangle}X_{\nu(c)}\\
&=2\overline{\langle x^{\frac{2n}{2^{\nu(c)}}}\rangle}X_{\nu(c)-1}+\left(\overline{\langle x^{\frac{2n}{2^{\nu(c)}}}\rangle}+x^{\frac{n}{2^{\nu(c)}}}\overline{\langle x^{\frac{2n}{2^{\nu(c)}}}\rangle}\right)X_{\nu(c)}\\
&=\left(2X_{\nu(c)-1}+X_{\nu(c)}+x^{\frac{n}{2^{\nu(c)}}}X_{\nu(c)}\right)\overline{\langle x^{\frac{2n}{2^{\nu(c)}}}\rangle}\\
&\overset{\text{def}}=\overline{x^{E_0}}\;\overline{\langle x^{\frac{2n}{2^{\nu(c)}}}\rangle},
\end{aligned}
\end{equation*}
where $E_0$ is a multisubset of $\{0,1,\cdots,{\frac{2n}{2^{\nu(c)}}-1}\}$.\;Therefore
\begin{equation*}
\begin{aligned}
U=x^{E_0}\langle x^{\frac{2n}{2^{\nu(c)}}}\rangle.
\end{aligned}
\end{equation*}
The result follows.
\end{proof}
Let $\omega_o(c)$ denote the number of distinct odd prime divisors of $c$,\;i.e.
\[\omega_o(c)=\#\{p:p\in \mathbb{P},p>2,p|c\},\]
where $ \mathbb{P}$ is the set of all the prime numbers.\;The following lemmas give a more detailed structure of $U$ if
\[\chi(\overline{U})\equiv 0 \;\mathrm{mod}\;c\]
for all nonprincipal characters $\chi$ of $C_n$.\;

\begin{lem}\label{l-3.4}Let $c$ be an odd number and $U$ be a multisubset of $C_n$ with $\Delta_{U}\leqslant2$.\;If
\[\chi(\overline{U})\equiv 0 \;\mathrm{mod}\;c\]
for all nonprincipal characters $\chi$ of $C_n$.\;Then there exists a multisubset $E$ of $\{0,1,\cdots,{\frac{n}{c}-1}\}$ such that the multiplicity function $\Delta_{E}\leqslant2$ and
\begin{equation}
\begin{aligned}
U=x^{E}\langle x^{\frac{n}{c}}\rangle.
\end{aligned}
\end{equation}
\end{lem}
\begin{proof}The proof proceeds by induction on $\omega_o(c)$.\;This assertion for $\omega_o(c)=1$ has been proved in Lemma \ref{l-3.2}.\;Now let $c=dp^s$ with odd prime $p\nmid d$,\;then $\omega_o(c)=\omega_o(d)+1$.\;Based on the assumption of this lemma,\;we can get
\[\chi(\overline{U})\equiv 0 \;\mathrm{mod}\;d\]
for all nonprincipal characters $\chi$ of $C_n$.\;Then by applying the induction hypothesis,\;there is a multisubset $E'$ of $\{0,1,\cdots,\frac{n}{d}-1\}$ such that the multiplicity function $\Delta_{E'}\leqslant2$ and
\begin{equation*}
\begin{aligned}
U=x^{E'}\langle x^{\frac{n}{d}}\rangle\overset{\text{def}}={T'}\langle x^{\frac{n}{d}}\rangle.
\end{aligned}
\end{equation*}
Let $\pi_{\frac{n}{d}}$ be the natural projection from $C_n=\langle x\rangle$ to the quotient group $\langle x\rangle/\langle x^{\frac{n}{d}}\rangle=\langle \pi_{\frac{n}{d}}(x)\rangle$,\;this is a cyclic group of order $\frac{n}{d}$.\;Let
$\widetilde{x}=\pi_{\frac{n}{d}}(x)$,\;then $$\pi_{\frac{n}{d}}(U)=\pi_{\frac{n}{d}}({T'}\langle x^{\frac{n}{d}}\rangle)=d\oplus \pi_{\frac{n}{d}}(T').$$
From Lemma \ref{l-2.8},\;there is a bijection between
\[\widehat{\langle x\rangle/\langle x^{\frac{n}{d}}\rangle}\leftrightarrow\{\chi\in \widehat{C_n}|\langle x^{\frac{n}{d}}\rangle\subseteq\ker\chi\}=\{\chi_{dj}|0\leqslant j\leqslant{n}/{d}-1\}.\]
Recall that $\widehat{\langle x\rangle/\langle x^{\frac{n}{d}}\rangle}=\{\chi'_j|0\leqslant j\leqslant\frac{n}{d}-1\}$,\;where $\chi'_j$ is the character of quoient group $\langle x\rangle/\langle x^{\frac{n}{d}}\rangle$ induced by the character $\chi_{dj}\in \widehat{C_n}$,\;
for any $0\leqslant j\leqslant\frac{n}{d}-1$.\;Therefore,\;for all $1\leqslant j\leqslant\frac{n}{d}-1$,\;
\begin{equation}\label{3.4}
\begin{aligned}
\chi_j'(d\oplus \pi_{\frac{n}{d}}(T'))=d\chi_j'(\pi_{\frac{n}{d}}(T'))=d\chi_{dj}(T')=\chi_{dj}({T'}\langle x^{\frac{n}{d}}\rangle)=\chi_{dj}(\overline{U})\equiv 0 \;\mathrm{mod}\;c.
\end{aligned}
\end{equation}
This gives that
\[\chi_j'(\pi_{\frac{n}{d}}(T'))\equiv 0 \;\mathrm{mod}\;p^s\]
for any $1\leqslant j\leqslant\frac{n}{d}-1$.\;Then from Lemma  \ref{l-3.2},\;there is multisubset $E$ of $\{0,1,\cdots,{\frac{n}{dp^s}-1}\}=\{0,1,\cdots,{\frac{n}{c}-1}\}$ such that the multiplicity function $\Delta_{E}\leqslant2$ and
\begin{equation*}
\begin{aligned}
\pi_{\frac{n}{d}}(T')=\widetilde{x}^{E}\langle \widetilde{x}^{\frac{n}{dp^s}}\rangle=\{\widetilde{x}^{a+\frac{n}{dp^s}j}:a\in E,0\leqslant j\leqslant p^s-1\}=\bigcup_{a\in E}\bigcup_{j=0}^{p^s-1}\{\pi_{\frac{n}{d}}(x^{a+\frac{n}{dp^s}j})\}.
\end{aligned}
\end{equation*}
Since $\ker\pi_{\frac{n}{d}}=\langle x^{\frac{n}{d}}\rangle$,\;the above equation gives that
\begin{equation*}
\begin{aligned}
T'=\bigcup_{a\in E}\bigcup_{j=0}^{p^s-1}\{x^{a+\frac{n}{dp^s}j}b_{a,j}\}
\end{aligned}
\end{equation*}
for some $b_{a,j}\in\langle x^{\frac{n}{d}}\rangle$,\;where $a\in E$.\;Then we conclude that
\begin{equation*}
\begin{aligned}
U={T'}\langle x^{\frac{n}{d}}\rangle=\bigcup_{a\in E}\bigcup_{j=0}^{p^s-1}\{x^{a+\frac{n}{dp^s}j}b_{a,j}\langle x^{\frac{n}{d}}\rangle\}
=x^E\left(\bigcup_{j=0}^{p^s-1}x^{\frac{n}{dp^s}j}\langle x^{\frac{n}{d}}\rangle\right)=x^E\langle x^{\frac{n}{c}}\rangle.
\end{aligned}
\end{equation*}
\end{proof}
\begin{lem}\label{l-3.5}Let $c=2^{\nu(c)}c_1$ be an even number and $U$ be a multisubset of $C_n$ with $\Delta_{U}\leqslant2$.\;If
\[\chi(\overline{U})\equiv 0 \;\mathrm{mod}\;c\]
for all nonprincipal characters $\chi$ of $C_n$.\;Then there is a multisubset $E$ of $\{0,1,\cdots,{\frac{2n}{c}-1}\}$ such that the multiplicity function $\Delta_{E}\leqslant2$ and
\begin{equation}
\begin{aligned}
U=x^{E}\langle x^{\frac{2n}{c}}\rangle,
\end{aligned}
\end{equation}
\end{lem}
\begin{proof}Observe that
\[\chi(\overline{U})\equiv 0 \;\mathrm{mod}\;c\Rightarrow\chi(\overline{U})\equiv 0 \;\mathrm{mod}\;c_1\]
for all nonprincipal characters $\chi$ of $C_n$.\;Then by Lemma \ref{l-3.4},\;there is multisubset $E'$ of $\{0,1,\cdots,{\frac{n}{c_1}-1}\}$ such that the multiplicity function $\Delta_{E'}\leqslant2$ and
\begin{equation}
\begin{aligned}
U=x^{E'}\langle x^{\frac{n}{c_1}}\rangle\overset{\text{def}}={T'}\langle x^{\frac{n}{c_1}}\rangle.
\end{aligned}
\end{equation}
Let $r=2^{\nu(c)}$ and $\pi_{\frac{n}{c_1}}$ be the natural projection from $C_n=\langle x\rangle$ to the quotient group $\langle x\rangle/\langle x^{\frac{n}{c_1}}\rangle=\langle\pi_{\frac{n}{c_1}}(x)\rangle$ which is a cyclic group of order $\frac{n}{c_1}$.\;Let
$\widetilde{x}=\pi_{\frac{n}{c_1}}(x)$,\;then
$$\pi_{\frac{n}{c_1}}(U)=\pi_{\frac{n}{c_1}}({T'}\langle x^{\frac{n}{c_1}}\rangle)=c_1\oplus \pi_{\frac{n}{c_1}}(T').$$
Similar to the calculation of (\ref{3.4}),\;we also have
\[\chi'(\pi_{\frac{n}{c_1}}(T'))\equiv 0 \;\mathrm{mod}\;r\]
for any characters $\chi'$  of quotient group $\langle x\rangle/\langle x^{\frac{n}{c_1}}\rangle$.\;Then the Lemma \ref{l-3.3} implies that,\;there is a multisubset $E$ of $\{0,1,\cdots,{\frac{2n}{c_1r}-1}\}$ such that the multiplicity function $\Delta_{E}\leqslant2$ and
\begin{equation*}
\begin{aligned}
\pi_{\frac{n}{c_1}}(T')=\widetilde{x}^{E}\langle \widetilde{x}^{\frac{2n}{c_1r}}\rangle=\bigcup_{a\in E}\bigcup_{j=0}^{\frac{r}{2}-1}\{\pi_{\frac{n}{c_1}}(x^{a+\frac{2n}{c_1r}j})\}.
\end{aligned}
\end{equation*}
This shows that
\begin{equation*}
\begin{aligned}
T'=\bigcup_{a\in E}\bigcup_{j=0}^{\frac{r}{2}-1}\{x^{a+\frac{2n}{c_1r}j}b_{a,j}\}
\end{aligned}
\end{equation*}
for some $b_{a,j}\in\langle x^{\frac{n}{c_1}}\rangle$,\;where $a\in E$.\;Then we  conclude that
\begin{equation*}
\begin{aligned}
U={T'}\langle x^{\frac{n}{c_1}}\rangle=\bigcup_{a\in E}\bigcup_{j=0}^{\frac{r}{2}-1}\{x^{a+\frac{2n}{c_1r}j}b_{a,j}\langle x^{\frac{n}{c_1}}\rangle\}
=x^E\left(\bigcup_{j=0}^{\frac{r}{2}-1}x^{\frac{2n}{c_1r}j}\langle x^{\frac{n}{c_1}}\rangle\right)=x^E\langle x^{\frac{2n}{c}}\rangle.
\end{aligned}
\end{equation*}
\end{proof}

\subsection{The main lemmas occur in the proof of the main Theorem \ref{t-1.4}}
Throughout this section,\;let $c$ be an integer,\;not necessary be positive.\;We now assume $U$ is an unempty multisubset of $C_n$ such that $e\not\in U$ and $\Delta_U\leqslant 2$,\;then there is a multisubset $E$ of $\mathbb{Z}_n\setminus\{0\}$ such that $U=x^E$,\;and in addition,\;suppose $U$ satisfy $$\chi(\overline{U})\in\{0,c\},\forall\;\chi\neq\chi_0,\;i.e.\;(\mathcal{F}\Delta_E)(z)\in\{0,c\},\forall \;0\neq z\in\mathbb{Z}_n.$$\;
We define $\Gamma_c=\{z:z\in \mathbb{Z}_n,\;z\neq0,\;\chi_z(\overline{U})=(\mathcal{F}\Delta_E)(z)=c\}$.\;Let $\underset{i\in\Gamma_{c}}{\mathbf{gcd}}\left\{i\right\}$ denote the greatest common divisor of elements in $\Gamma_{c}$.\;
We define
\begin{equation}\label{3.7}
\begin{aligned}
\delta_c=\mathbf{gcd}(n,\underset{i\in\Gamma_{c}}{\mathbf{gcd}}\left\{i\right\})\text{\;and\;}S_c=\{v:v|n,v\nmid \delta_c\},
\end{aligned}
\end{equation}
where the elements of $\Gamma_c$ are viewed as the integers from the set $\{1,2,\cdots,n-1\}$.\;Then $1\leqslant \delta_c\leqslant n-1 $.\;We have the following lemma.\;

\begin{lem}\label{l-3.6}Let $U=x^E$ be a multisubset of $C_n$ such that $e\not\in U$,\;$\Delta_U\leqslant 2$ and \[\chi(\overline{U})\in\{0,c\}\]
for all nonprincipal characters $\chi$,\;then:\\
$(1)$\;There are some intergers $1<r_1<r_2<\cdots<r_s$ and $1<r_{s+1}<r_{s+2}<\cdots<r_{t}$ satisfy
\begin{equation}\label{3.8}
E=(2\oplus \mathcal{O}_{r_1})\uplus (2\oplus \mathcal{O}_{r_2})\uplus\cdots\uplus (2\oplus \mathcal{O}_{r_s})\uplus (\mathcal{O}_{r_{s+1}}\cup \mathcal{O}_{r_{s+2}}\cup\cdots\cup \mathcal{O}_{r_t}),
\end{equation}
where $r_1,r_2,\cdots,r_t$ are divisors of $n$.\;Therefore,\;$E=-E$,\;i.e.,\;$\Delta_E(z)=\Delta_E(-z)$ for each $z\in E$.\\
$(2)$\:$\mathbb{Z}_{n}\setminus E=\frac{n}{\delta_c}\mathbb{Z}_{n}$.\\
$(3)$\;$|c|$ is a divisor of $n$.\\
$(4)$\;Let $\mathcal{I}_1=\{r_1,r_2,\cdots,r_s\},\;\mathcal{I}_2=\{r_{s+1},r_{s+2},\cdots,r_{t}\}$ throughout this section,\;then $\mathcal{I}_1$ and $\mathcal{I}_2$ form a partition of $S_c$.\;Therefore
\begin{equation}\label{3.9}
\begin{aligned}
E&=\left(\bigcup_{j=1}^s\mathcal{O}_{r_j}\right)\uplus\left(\bigcup_{j\in S_c}\mathcal{O}_{j}\right)
\end{aligned}
\end{equation}
and
\begin{equation}\label{3.10}
\begin{aligned}
(\mathcal{F}\Delta_{E})(z)=\sum_{j=1}^s\mu\left(\frac{r_i}{(r_i,z)}\right)\frac{\varphi(r_i)}{\varphi\left(\frac{r_i}{(r_i,z)}\right)}
+n\Delta_{0}(z)-\delta_c\Delta_{\delta_c\mathbb{Z}_{n}}(z).
\end{aligned}
\end{equation}
\end{lem}
\begin{proof}Note that $$\mathrm{Im}(\mathcal{F}\Delta_E)\in \mathbb{Q}.$$\;Therefore,\;
\[\Delta_{E}=\sum_{r|n}\alpha_r\Delta_{\mathcal{O}_r}\]
for some $\alpha_r\in\{0,1,2\}$ by Lemma \ref{l-2.10}.\;Note that $0\not\in E$,\;so $\alpha_0=0$.\;Then,\;there are some intergers $1\leqslant r_1<r_2<\cdots<r_s\leqslant \alpha$ and $1\leqslant r_{s+1}<r_{s+2}<\cdots<r_{t}\leqslant \alpha$ satisfy
\begin{equation*}
E=(2\oplus \mathcal{O}_{r_1})\uplus (2\oplus \mathcal{O}_{r_2})\uplus\cdots\uplus (2\oplus \mathcal{O}_{r_s})\uplus (\mathcal{O}_{r_{s+1}}\cup \mathcal{O}_{r_{s+2}}\cup\cdots\cup \mathcal{O}_{r_t}).
\end{equation*}
proving $(1)$.\;Secondly,\;by the inversion formula (\ref{2.3}) of Fourier transformation,\;we have
\begin{equation}\label{3.11}
\begin{aligned}
n\Delta_E(z)&=(\mathcal{F}(\mathcal{F}\Delta_E))(-z)=c\sum_{i\in \Gamma_{c}}\omega^{-iz}+|E|=c(\mathcal{F}\Delta_{\Gamma_c})(-z)+|E|,\\
\end{aligned}
\end{equation}
this gives that $\mathbf{Im}(\mathcal{F}(\Delta_{\Gamma_c}))\in\mathbb{Q}$ and hence $\mathbf{Im}(\mathcal{F}(\Delta_{\Gamma_c}))\in\mathbb{Z}$ from Lemma \ref{l-2.10}.\;Note that $\Delta_E(0)=c|\Gamma_{c}|+|E|=0$,\;so $|E|=-c|\Gamma_{c}|$.\;Then (\ref{3.11}) becomes
\begin{equation}\label{3.12}
\begin{aligned}
n\Delta_E(z)=c((\mathcal{F}(\Delta_{\Gamma_c}))(-z)-|\Gamma_{c}|),\\
\end{aligned}
\end{equation}
this asserts that $|c|$ is a divisor of $n$,\;proving $(3)$.\;Meanwhile,\;from (\ref{3.11}) and (\ref{3.12}),\;we can get
\begin{equation*}
\begin{aligned}
z\in\mathbb{Z}_{n}\setminus E &\Leftrightarrow  \Delta_E(z)=0
\Leftrightarrow \sum_{i\in \Gamma_{c}}(1-\omega^{-iz})=0 \Leftrightarrow \omega^{-iz}=1,\forall i\in\Gamma_{c}\Leftrightarrow z\in \bigcap_{i\in\Gamma_{c}}\left(\frac{n}{(n,i)}\mathbb{Z}_{n}\right).
\end{aligned}
\end{equation*}
Therefore
\[\mathbb{Z}_{n}\setminus E=\bigcap_{i\in\Gamma_{c}}\left(\frac{n}{(n,i)}\mathbb{Z}_{n}\right)=\left(\underset{i\in\Gamma_{c}}{\mathbf{lcm}}\left\{\frac{n}{(n,i)}\right\}\right)\mathbb{Z}_{n}=
\left(\frac{n}{(n,\underset{i\in\Gamma_{c}}{\mathbf{gcd}}\left\{i\right\})}\right)\mathbb{Z}_{n}=\frac{n}{\delta_c}\mathbb{Z}_{n},\]
proving $(2)$.\;Part $(4)$ follows from $(1)$,\;$(2)$ and $(3)$ directly.\;Indeed,\;note that
\[\frac{n}{\delta_c}\mathbb{Z}_{n}=\bigcup_{h|\delta_c}\mathcal{O}_h,\]
hence $\mathcal{I}_1$ and $\mathcal{I}_2$ form a partition of $S_c$.\;Therefore
\begin{equation*}
\begin{aligned}
E&=\left(\biguplus_{j=1}^s2\oplus\mathcal{O}_{r_j}\right)\uplus\left(\bigcup_{j=s+1}^t\mathcal{O}_{r_j}\right)=\left(\bigcup_{j=1}^s\mathcal{O}_{r_j}\right)\uplus\left(\bigcup_{j\in S_U}\mathcal{O}_{j}\right).
\end{aligned}
\end{equation*}
Then from (\ref{2.7}),\;
\begin{equation*}
\begin{aligned}
(\mathcal{F}\Delta_{E})(z)&=\sum_{j=1}^s(\mathcal{F}\Delta_{O_{r_j}})(z)+\sum_{j\in S_c}(\mathcal{F}\Delta_{O_{j}})(z)\\
&=\sum_{j=1}^s(\mathcal{F}\Delta_{O_{r_i}})(z)+(\mathcal{F}\Delta_{\mathbb{Z}_{n}})(z)-\sum_{j|\delta_c}(\mathcal{F}\Delta_{O_{j}})(z)\\
&=\sum_{j=1}^s\mu\left(\frac{r_i}{(r_i,z)}\right)\frac{\varphi(r_i)}{\varphi\left(\frac{r_i}{(r_i,z)}\right)}
+n\Delta_{0}(z)-\delta_c\Delta_{\delta_c\mathbb{Z}_{n}}(z),
\end{aligned}
\end{equation*}
proving $(4)$.\;
\end{proof}
\begin{rmk}$U$ is a subset of $C_n$ (i.e.,\;not a multisubset) $\Leftrightarrow$ $\mathcal{I}_1=\emptyset$.\;
\end{rmk}

When $c=-1$,\;we can determine the structure of $U$.\;
\begin{lem}\label{l-3.8}Let $U=x^E$ be a multisubset of $C_n$ such that $e\not\in U$,\;$\Delta_U\leqslant2$ and
\[\chi(\overline{U})\in\{0,-1\}\]
for all nonprincipal characters $\chi$,\;then $U=C_n\setminus\{e\}$.\;
\end{lem}
\begin{proof}At first,\;$E=-E$ from Lemma \ref{l-3.6} $(1)$.\;We first assert that $\delta_{-1}=1$.\;We divide into two cases.\;\\
$\mathbf{Case\;1}$.\;If $\mathcal{I}_1=\emptyset$.\;Then from Lemma \ref{l-3.6} $(4)$,\;
\begin{equation*}
\begin{aligned}
(\mathcal{F}\Delta_{E})(z)=n\Delta_{0}(z)-\delta_{-1}\Delta_{\delta_{-1}\mathbb{Z}_{n}}(z).
\end{aligned}
\end{equation*}
Note that $1\leqslant \delta_{-1}\leqslant n-1$ and hence $\Delta_{0}(\delta_{-1})=0$.\;Then we have
\begin{equation*}
\begin{aligned}
(\mathcal{F}\Delta_{E})(\delta_{-1})=-\delta_{-1}\in\{0,-1\}.
\end{aligned}
\end{equation*}
This implies that $\delta_{-1}=1$.\;\\
$\mathbf{Case\;2}$.\;If $\mathcal{I}_1\neq\emptyset$.\;Then the (\ref{3.12}) becomes
\begin{equation}
\begin{aligned}
n\Delta_E(z)=-((\mathcal{F}\Delta_{\Gamma_{-1}})(-z)-|\Gamma_{-1}|)\Rightarrow(\mathcal{F}\Delta_{\Gamma_{-1}})(z)=|\Gamma_{-1}|-n\Delta_E(-z).
\end{aligned}
\end{equation}
This shows that  $\mathrm{Im}(\mathcal{F}(\Delta_{\Gamma_{-1}}))\in \mathbb{Q}$.\;Therefore Lemma \ref{l-2.10} implies that
\[\Gamma_{-1}=\mathcal{O}_{{c_1}}\cup \mathcal{O}_{{c_2}}\cup\cdots\cup \mathcal{O}_{{c_l}}\]
for some $1<c_1,c_2,\cdots,c_l\leqslant n$.\;Then $\delta_{-1}=\gcd(\frac{n}{c_1},\frac{n}{c_2},\cdots,\frac{n}{c_l})$ by definition (\ref{3.7}),\;and
\begin{equation*}
  (\mathcal{F}\Delta_{\Gamma_{-1}})(z)=\sum_{j=1}^l\mu\left(\frac{c_j}{(c_j,z)}\right)\frac{\varphi(c_j)}{\varphi\left(\frac{c_j}{(c_j,z)}\right)}=|\Gamma_{-1}|-n\Delta_E(-z)
\end{equation*}
by Lemma \ref{l-3.6} $(4)$.\;Since $|\Gamma_{-1}|=\sum_{j=1}^l\varphi(c_j)$,\;then
\begin{equation}
\sum_{j=1}^l\varphi(c_j)\left(1-\frac{\mu\left(\frac{c_j}{(c_j,z')}\right)}{\varphi\left(\frac{c_j}{(c_j,z')}\right)}\right)={n}\Delta_E(-z).
\end{equation}
Furthermore,\;we can choose some  $z'\in E$ such that $\Delta_E(z')=2$,\;since $\mathcal{I}_1\neq\emptyset$.\;But
\[2n={n}\Delta_E(-z')=\sum_{j=1}^l\varphi(c_j)\left(1-\frac{\mu\left(\frac{c_j}{(c_j,z')}\right)}{\varphi\left(\frac{c_j}{(c_j,z')}\right)}\right)\leqslant2\sum_{j=1}^l\varphi(c_j)=2|\Gamma_{-1}|<2n,\]
this is a contradiction.\;

Combining {${\mathbf{Case\;1}}$} and {${\mathbf{2}}$},\;we conclude that $\delta_{-1}=1$.\;It follows from Lemma \ref{l-3.6} $(2)$ that
$\mathbb{Z}_{n}\setminus E=\frac{n}{\delta_{-1}}\mathbb{Z}_{n}=\{0\}$,\;then $C_n\setminus\{e\}\subseteq U$ and hence $|U|\geqslant n-1$.\;On the other hand,\;noticing that,\;
\[-(n-1)\leqslant\sum_{\chi\neq\chi_0}\chi(\overline{U})=\sum_{\chi}\chi(\overline{U})-|U|=-|U|,\]
hence $|U|\leqslant n-1$.\;Therefore $U=C_n\setminus\{e\}$ holds.\;
\end{proof}

When $c=-2$,\;we can also determine the structure of $U$.\;
\begin{lem}\label{l-3.9}Let $U=x^E$ be a multisubset of $C_n$ such that $e\not\in U$,\;$\Delta_U\leqslant2$ and
\[\chi(\overline{U})\in\{0,-2\}\]
for all nonprincipal characters $\chi$,\;then $n$ is even and $U=C_n\setminus\{e,x^{\frac{n}{2}}\}$ or $U=\{x^{\frac{n}{2}}\}\uplus\left(C_n\setminus\{e\}\right)$.\;
\end{lem}
\begin{proof}The fact that $n$ is even follows from  Lemma \ref{l-3.6} $(3)$.\;To prove the next assertion,\;we divide into two cases.\;\\
$\mathbf{Case\;1}$.\;If $\mathcal{I}_1=\emptyset$.\;Then from Lemma \ref{l-3.6} $(4)$,\;$(\mathcal{F}\Delta_{E})(z)=-\delta_{-2}\Delta_{\delta_{-2}\mathbb{Z}_{n}}(z)\in\{0,-2\}$ for $0\neq z$.\;Thus  $\delta_{-2}=2$  and
\begin{equation*}
\begin{aligned}
E=\bigcup_{j|n,j\neq 1,2}\mathcal{O}_{j}=\mathbb{Z}_n\setminus\left\{0,\frac{n}{2}\right\}.
\end{aligned}
\end{equation*}
Hence $U=x^E=C_n\setminus\{e,x^{\frac{n}{2}}\}$.\;\\
$\mathbf{Case\;2}$.\;If $\mathcal{I}_1\neq\emptyset$.\;The (\ref{3.11}) becomes
\begin{equation}\label{3.15}
\begin{aligned}
n\Delta_E(z)=-2(\mathcal{F}\Delta_{\Gamma_{-2}})(-z)+|E|\Rightarrow(\mathcal{F}\Delta_{\Gamma_{-2}})(z)=\frac{|E|}{2}-\frac{n}{2}\Delta_E(-z).
\end{aligned}
\end{equation}
This shows that  $\mathrm{Im}(\mathcal{F}(\Delta_{\Gamma_{-1}}))\in \mathbb{Q}$.\;Thus Lemma \ref{l-2.10} also implies that
\[\Gamma_{-2}=\mathcal{O}_{{c_1}}\cup \mathcal{O}_{{c_2}}\cup\cdots\cup \mathcal{O}_{{c_l}}\]
for some $1<c_1,c_2,\cdots,c_l\leqslant n$.\;Then $\delta_{-2}=\gcd(\frac{n}{c_1},\frac{n}{c_2},\cdots,\frac{n}{c_l})$ by definition (\ref{3.7}),\;and
\begin{equation*}
(\mathcal{F}\Delta_{\Gamma_{-2}})(z)=\sum_{j=1}^l\mu\left(\frac{c_j}{(c_j,z)}\right)\frac{\varphi(c_j)}{\varphi\left(\frac{c_j}{(c_j,z)}\right)}=|\Gamma_{-2}|-\frac{n}{2}\Delta_E(-z).
\end{equation*}
Therefore,\;
\begin{equation}\label{3.16}
\sum_{j=1}^l\varphi(c_j)\left(1-\frac{\mu\left(\frac{c_j}{(c_j,z)}\right)}{\varphi\left(\frac{c_j}{(c_j,z)}\right)}\right)=\frac{n}{2}\Delta_E(-z).
\end{equation}
$\mathbf{Case\;2.1}$.\;$\Delta_E(\frac{n}{2})=0$.\;Therefore $\frac{n}{\delta_{-2}}|\frac{n}{2}$ from Lemma \ref{l-3.6} $(2)$,\;and then $2|\delta_{-2}$.\;It follows from the definition $\delta_{-2}=\gcd(\frac{n}{c_1},\frac{n}{c_2},\cdots,\frac{n}{c_l})$ that $c_j|\frac{n}{2}$ for each $1\leqslant j\leqslant l$.\;Since $\mathcal{I}_1\neq\emptyset$,\;selecting some $z'\in E$ such that $\Delta_E(z')=2$.\;Then
\[\frac{n}{2}\Delta_E(-z')=n=\sum_{j=1}^l\varphi(c_j)\left(1-\frac{\mu\left(\frac{c_j}{(c_j,z')}\right)}{\varphi\left(\frac{c_j}{(c_j,z')}\right)}\right)\leqslant2\sum_{j=1}^l\varphi(c_j)\leqslant2\sum_{d|\frac{n}{2}}\varphi(d)=n,\]
from (\ref{3.16}).\;This implies that $\{c_1,c_2,\cdots,c_l\}=\{d:1\leqslant d \leqslant\frac{n}{2},\;d|\frac{n}{2}\}$ and
\[\Gamma_{-2}=\mathcal{O}_{{c_1}}\cup \mathcal{O}_{{c_2}}\cup\cdots\cup \mathcal{O}_{{c_l}}=\bigcup_{d|\frac{n}{2}}\mathcal{O}_{d}=2\mathbb{Z}_n.\]
Therefore
\begin{equation}
\begin{aligned}
\Delta_E(z)=\frac{1}{n}\left(-2(\mathcal{F}\Delta_{\Gamma_{-2}})(-z)+|E|\right)=\frac{|E|}{n}-\Delta_{\frac{n}{2}\mathbb{Z}_n}(z).
\end{aligned}
\end{equation}
This implies $|E|=n$ since $\Delta_E(\frac{n}{2})=0$.\;But $2=\Delta_E(z')=1-\Delta_{\frac{n}{2}\mathbb{Z}_n}(z')<2$,\;which leads to a contradiction.\\
$\mathbf{Case\;2.2}$.\;$\Delta_E(\frac{n}{2})=1$.\;In this case,\;(\ref{3.16}) implies that
\[\sum_{\substack{j=1\\c_j\nmid\frac{n}{2}}}^l\varphi(c_j)=\frac{n}{4}.\]
Let $n=2^{\nu(n)}n_1$,\;then $c_j\nmid\frac{n}{2}$ implies that $c_j=2^{\nu(n)}d_j$ for some $d_j|n_1$ and $2\nmid d_j$.\;Therefore
\[\frac{n}{4}=2^{\nu(n)-2}n_1=\sum_{\substack{j=1\\c_j\nmid\frac{n}{2}}}^l\varphi(2^{\nu(n)}d_j)=2^{\nu(n)-1}\sum_{\substack{j=1\\c_j\nmid\frac{n}{2}}}^l\varphi(d_j)\]
This give that
\[n_1=2\sum_{\substack{j=1\\c_j\nmid\frac{n}{2}}}^l\varphi(d_j),\]
contradicting to $2\nmid n_1$.\\
$\mathbf{Case\;2.3}$.\;$\Delta_E(\frac{n}{2})=2$.\;In this case,\;(\ref{3.16}) implies that
\[\sum_{\substack{j=1\\c_j\nmid\frac{n}{2}}}^l\varphi(c_j)=\frac{n}{2}=\sum_{\substack{d|n\\d\nmid\frac{n}{2}}}^l\varphi(d).\]
This asserts that $\{c_1,c_2,\cdots,c_l\}=\{d:d|n,\;d\nmid\frac{n}{2}\}$ and then
\[\Gamma_{-2}=\mathcal{O}_{{c_1}}\cup \mathcal{O}_{{c_2}}\cup\cdots\cup \mathcal{O}_{{c_l}}=\mathbb{Z}_n\setminus\left(\bigcup_{d|\frac{n}{2}}\mathcal{O}_{d}\right)=\mathbb{Z}_n\setminus2\mathbb{Z}_n.\]
Therefore
\begin{equation}
\begin{aligned}
\Delta_E(z)=\frac{1}{n}\left(-2(\mathcal{F}\Delta_{\Gamma_{-2}})(-z)+|E|\right)=\frac{|E|}{n}-2\Delta_0(z)+\Delta_{\frac{n}{2}\mathbb{Z}_n}(z).
\end{aligned}
\end{equation}
Since $\Delta_E(\frac{n}{2})=2$,\;then $|E|=n$.\;Therefore
$$E=\left\{\frac{n}{2}\right\}\uplus\left(\mathbb{Z}_n\setminus\{0\}\right)$$
and
$$U=\{x^{\frac{n}{2}}\}\uplus\left(C_n\setminus\{e\}\right).$$\;
This completes the proof.\;
\end{proof}
\section{Basic lemmas of directed strongly regular dihedrant $Dih(n,X,X)$}
For any subset $X$ of $C_n$,\;we define $U_X=X\uplus X^{(-1)}$,\;where $X^{(-1)}=\{g^{-1}|g\in X\}$.\;We now give a sufficient and necessary condition for the dihedrant $Dih(n,X,Y)$ to be directed strongly regular.\;
\begin{lem}\label{l-4.1}The dihedrant $Dih(n,X,Y)$ is a DSRG with parameters $(2n,|X|+|Y|, \mu, \lambda, t)$ if and only if $X$ and $Y$ satisfy the following conditions:
\begin{flalign}
(i)\;&\overline{Y}\;\overline{U_X}=(\lambda-\mu)\overline{Y}+\mu\overline{C_n};\hspace{270pt}\label{4.1}\\
(ii)\;&\overline{X}^2+\overline{Y}\;\overline{Y^{-1}}=(t-\mu)e+(\lambda-\mu)\overline{X}+\mu\overline{C_n}.\label{4.2}
\end{flalign}
\end{lem}
\begin{proof}Note that
\begin{equation*}
\begin{aligned}
(\overline{X}+\overline{Y\tau})^2&=\overline{{X}}\;\overline{{X}}+\overline{X}\;\overline{Y\tau}+\overline{Y\tau}\;\overline{X}+\overline{Y\tau}\;\overline{Y\tau}=\overline{{X}}\;\overline{{X}}+\overline{{Y}}\;\overline{Y^{(-1)}}+(\overline{{X}}\;\overline{{Y}}+\overline{{Y}}\;\overline{{X^{(-1)}}})\tau\\
&=(\overline{Y}\;\overline{U_X})\tau+(\overline{X}^2+\overline{Y}\;\overline{Y^{-1}}).
\end{aligned}
\end{equation*}
Thus,\;from Lemma \ref{l-2.6},\;the dihedrant $Dih(n,X,Y)$ is a DSRG with parameters $(2n,|X|+|Y|, \mu, \lambda, t)$ if and only if
\begin{equation*}
\begin{aligned}
(\overline{Y}\;\overline{U_X})\tau+(\overline{X}^2+\overline{Y}\;\overline{Y^{-1}})&=te+\lambda(\overline{X}+\overline{Y\tau})+\mu(\overline{D_n}-(\overline{X}+\overline{Y\tau})-e)\\
&=(t-\mu)e+(\lambda-\mu)\overline{X}+\mu\overline{C_n}+((\lambda-\mu)\overline{Y}+\mu\overline{C_n})\tau\\
\end{aligned}
\end{equation*}
This  is equivalent to the conditions $(\ref{4.1})$ and $(\ref{4.2})$.\;
\end{proof}
When $Y=X$,\;we have the following lemma.\;
\begin{lem}\label{l-4.2}The dihedrant $Dih(n,X,X)$ is a DSRG with parameters $(2n,2|X|, \mu, \lambda, t)$ if and only if $t=\mu$ and $X$ satisfy the following conditions:
\begin{equation}
\overline{X}\;\overline{U_X}=(\lambda-\mu)\overline{X}+\mu\overline{C_n}.
\end{equation}
\end{lem}
Define the notation
\[\delta_{P,Q}=\left\{
  \begin{array}{ll}
    1, & \hbox{$P=Q$;} \\
    0, & \hbox{$P\neq Q$.}
  \end{array}
\right.\]
The following lemma give a sufficient and necessary condition for the dihedrant $Dih(n,X,X)$ to be directed strongly regular,\;by using the characters of $C_n$.\;
\begin{lem}\label{l-4.3}The dihedrant $Dih(n,X,X)$ is a DSRG with parameters $(2n,2|X|, \mu, \lambda, t)$ if and only if
\begin{equation}
\chi(\overline{X})\chi(\overline{U_X})=(\lambda-\mu)\chi(\overline{X})+\mu n\delta_{\chi,\chi_0}
\end{equation}
hold for any $\chi\in \widehat{C_n}$.
\end{lem}

We now give another version of Lemma \ref{l-4.1}.\;Let $X=x^E$ and $Y=x^F$,\;where $E$ and $F$ are some multisubets of $\mathbb{Z}_n$.\;We now define
\[\mathbf{r}_E(z)=(\mathcal{F}\Delta_E)(z)=\sum_{i\in E}\zeta_n^{iz}\;\text{and}\;\mathbf{r}_F(z)=(\mathcal{F}\Delta_F)(z)=\sum_{i\in F}\zeta_n^{iz}.\]
Then $\mathbf{r}_E(z)+\overline{\mathbf{r}_E(z)}=(\mathcal{F}\Delta_{X\uplus(-X)})(z)$.\;The following lemma give a sufficient and necessary condition for the dihedrant $Dih(n,X,Y)$ to be directed strongly regular by using $\mathbf{r}_E(z)$ and $\mathbf{r}_F(z)$.\;
\begin{lem}\label{l-4.4}The dihedrant $Dih(n,X,Y)$ is a DSRG with parameters $(2n,|X|+|Y|, \mu, \lambda, t)$ if and only if $\mathbf{r}$ and $\mathbf{t}$ satisfy the following conditions:
\begin{flalign}
(i)\;&\mathbf{r}_F(\mathbf{r}_E+\overline{\mathbf{r}}_E)=\mu n\Delta_0+(\lambda-\mu)\mathbf{r}_F;\hspace{250pt}\label{4.5}\\
(ii)\;&\mathbf{r}_E^2+|\mathbf{r}_F|^2=t-\mu+\mu n\Delta_0+(\lambda-\mu)\mathbf{r}_F.\label{4.6}
\end{flalign}
\end{lem}
\begin{proof}By applying the Fourier transformation on (\ref{4.1}) and (\ref{4.2}).\;
\end{proof}
When $Y=X$,\;we have the following lemma.\;
\begin{lem}\label{l-4.5}The dihedrant $Dih(n,X,X)$ is a DSRG with parameters $(2n,2|X|, \mu, \lambda, t)$ if and only if $t=\mu$ and the function $\mathbf{r}_E$  satisfies
\begin{equation}\label{4.7}
\mathbf{r}_E(\mathbf{r}_E+\overline{\mathbf{r}_E})=\mu n\Delta_0+(\lambda-\mu)\mathbf{r}_E.
\end{equation}
\end{lem}

\begin{lem}\label{l-4.6}Let $Dih(n,X,X)$ be a directed strongly regular dihedrant with parameters $(2n,2|X|, \mu, \lambda, t)$,\;\\then
\begin{equation}\label{3.6}
\chi(\overline{U_X})\in\{0,\lambda-\mu\}
\end{equation}
for all nonprincipal characters $\chi$ of $C_n$ and $\mu-\lambda$ is a divisor of $n$.
\end{lem}
\begin{proof}From Lemma \ref{l-4.3},\;we have
\begin{equation*}
\chi(\overline{X})(\chi(\overline{X})+\overline{\chi(\overline{X})})=\chi(\overline{X})\chi(\overline{U_X})=(\lambda-\mu)\chi(\overline{X})+\mu n\delta_{\chi,\chi_0}
\end{equation*}
hold for any $\chi\in \widehat{C_n}$.\;We can get
\begin{equation*}
\overline{\chi(\overline{X})}(\chi(\overline{X})+\overline{\chi(\overline{X})})=\overline{\chi(\overline{X})}\chi(\overline{U_X})=(\lambda-\mu)\overline{\chi(\overline{X})}+\mu n\delta_{\chi,\chi_0}
\end{equation*}
by taking conjugate on the above equation.\;Adding these two expressions gives
\begin{equation*}
(\chi(\overline{U_X}))^2=(\lambda-\mu)\chi(\overline{U_X})+2\mu n\delta_{\chi,\chi_0}.
\end{equation*}
Therefore,\;
\begin{equation*}
\chi(\overline{U_X})\in\{0,\lambda-\mu\}.
\end{equation*}
for all nonprincipal characters $\chi$ of $C_n$.\;Then the fact that $\mu-\lambda$ is a divisor of $n$ now follows from Lemma \ref{l-3.6} $(3)$ directly.
\end{proof}
We also need the following lemma.\;
\begin{lem}\label{l-4.7}Let $E$ be a multisubset of $\mathbb{Z}_n$ such that
\[\mathbf{r}_E(z)=0\]
for all $z\not\in d\mathbb{Z}_n$.\;Then there is a multisubet $E'$ of $\{0,1,2,\cdots,d-1\}$ such that
\[E=E'+\frac{n}{d}\mathbb{Z}_n.\]
\end{lem}
\begin{proof}For any $a\in\frac{n}{d}\mathbb{Z}_n $,\;it follows from inverse formula (\ref{2.3}) that
\begin{equation*}
\begin{aligned}
n(\Delta_E(z+a)-\Delta_E(z))&=\sum_{i\in \mathbb{Z}_n}\mathbf{r}_E(i)(\omega^{-i(z+a)}-\omega^{-iz})\\
&=\sum_{i\in d\mathbb{Z}_n}\mathbf{r}_E(i)(\omega^{-i(z+a)}-\omega^{-iz})+\sum_{i\not\in d\mathbb{Z}_n}\mathbf{r}_E(i)(\omega^{-i(z+a)}-\omega^{-iz})\\
&=\sum_{i\in d\mathbb{Z}_n}\mathbf{r}_E(i)(\omega^{-i(z+a)}-\omega^{-iz})=0.
\end{aligned}
\end{equation*}
This shows that $E$ is a union of some cosets of $\frac{n}{d}\mathbb{Z}_n$ in $\mathbb{Z}_n$,\;proving this lemma.\;
\end{proof}
\begin{cor}\label{r-4.8}Let $T$ be a multisubset of $C_n$ such that
\[\chi_z(\overline{T})=0\]
for $z\not\in d\mathbb{Z}_n$.\;Then there is a multisubet $T'$ of $\{e,x^1,x^2,\cdots,x^{d-1}\}$ such that
\[T=T'\langle x^{\frac{n}{d}}\rangle.\]
\end{cor}

The following lemma gives a description of $X$ if $Dih(n,X,X)$ is a directed strongly regular dihedrant.\;
\begin{lem}\label{l-4.9}Let $Dih(n,X,X)$ be a directed strongly regular dihedrant with parameters $(2n,2|X|, \mu, \lambda, t)$.\;
If $\mu-\lambda$ is an odd number,\;then there is a subset $T$ of $\{x^1,\cdots,x^{\frac{n}{\mu-\lambda}-1}\}$ such that
\begin{equation}\label{4.9}
X={T}\langle x^{\frac{n}{\mu-\lambda}}\rangle.
\end{equation}
If $\mu-\lambda$ is an even number,\;then there is a subset $T'$ of $\{x^1,\cdots,x^{\frac{2n}{\mu-\lambda}-1}\}$ such that
\begin{equation}\label{4.10}
X={T'}\langle x^{\frac{2n}{\mu-\lambda}}\rangle.
\end{equation}
\end{lem}
\begin{proof}Let $X=x^E$ for some $E$,\;then $U_X=x^{E\uplus(-E)}$.\;Suppose $\mu-\lambda$ is an odd number.\;Recall that
\begin{equation*}
\chi(\overline{U_X})\in\{0,\lambda-\mu\}
\end{equation*}
for all nonprincipal characters $\chi$ of $G$.\;This gives that
\begin{equation*}
\chi(\overline{U_X})\equiv 0\;\mathrm{mod}\;\mu-\lambda.
\end{equation*}
for all nonprincipal characters $\chi$ of $G$.\;It follows  from Lemma \ref{l-3.4} that there is a multisubset $E''$ of $\{1,\cdots,{\frac{n}{\mu-\lambda}-1}\}$ such that the multiplicity function $\Delta_{E''}\leqslant2$ and
\begin{equation}
\begin{aligned}
U_X=x^{E''}\langle x^{\frac{n}{\mu-\lambda}}\rangle.
\end{aligned}
\end{equation}
Then
\begin{equation*}
\begin{aligned}
E\uplus(-E)={E''}+\frac{n}{\mu-\lambda}\mathbb{Z}_n.
\end{aligned}
\end{equation*}
Hence
\begin{equation}\label{4.12}
\begin{aligned}
(\mathcal{F}\Delta_{E\uplus(-E)})(z)=(\mu-\lambda)(\mathcal{F}\Delta_{E''})(z)\Delta_{(\mu-\lambda)\mathbb{Z}_n}(z).
\end{aligned}
\end{equation}
Therefore,\;from Lemma \ref{l-4.5} and (\ref{4.12}),\;(\ref{4.7}) becomes
\begin{equation}
\begin{aligned}
\mathbf{r}_E(z)(\mu-\lambda)(\mathcal{F}\Delta_{E''})(z)\Delta_{(\mu-\lambda)\mathbb{Z}_n}(z)=\mu n\Delta_0(z)-(\mu-\lambda)\mathbf{r}_E(z).
\end{aligned}
\end{equation}
Hence $\mathbf{r}_E(z)=0$ for any $z\not\in(\mu-\lambda)\mathbb{Z}_n$.\;It follows from Lemma \ref{l-4.7} that there exists a subset $E'$ of $\{1,\cdots,{\frac{n}{\mu-\lambda}-1}\}$ such that
\begin{equation*}
\begin{aligned}
E={E'}+\frac{n}{\mu-\lambda}\mathbb{Z}_n.
\end{aligned}
\end{equation*}
This shows that
\begin{equation*}
X=x^E=x^{E'}\langle x^{\frac{n}{\mu-\lambda}}\rangle\overset{\text{def}}=T\langle x^{\frac{n}{\mu-\lambda}}\rangle.
\end{equation*}
Now we consider the case of $\mu-\lambda$ is an even number.\;From Lemma \ref{l-3.5},\;we can also get that there is a multisubset $\widetilde{E}''$ of $\{1,\cdots,{\frac{2n}{\mu-\lambda}-1}\}$ such that the multiplicity function $\Delta_{\widetilde{E}''}\leqslant2$ and
\begin{equation}
\begin{aligned}
U_X=x^{\widetilde{E}''}\langle x^{\frac{2n}{\mu-\lambda}}\rangle.
\end{aligned}
\end{equation}
Then
\begin{equation*}
\begin{aligned}
E\uplus(-E)={\widetilde{E}''}+\frac{2n}{\mu-\lambda}\mathbb{Z}_n.
\end{aligned}
\end{equation*}
Hence
\begin{equation}\label{4.15}
\begin{aligned}
(\mathcal{F}\Delta_{E\uplus(-E)})(z)=\frac{\mu-\lambda}{2}(\mathcal{F}\Delta_{\widetilde{E}''})(z)\Delta_{\frac{\mu-\lambda}{2}\mathbb{Z}_n}(z).
\end{aligned}
\end{equation}
Now (\ref{4.7}) becomes
\begin{equation}
\begin{aligned}
\mathbf{r}_E(z)\frac{\mu-\lambda}{2}(\mathcal{F}\Delta_{\widetilde{E}''})(z)\Delta_{\frac{\mu-\lambda}{2}\mathbb{Z}_n}(z)=\mu n\Delta_0(z)-(\mu-\lambda)\mathbf{r}_E(z),
\end{aligned}
\end{equation}
from Lemma \ref{l-4.5} and (\ref{4.15}).\;Hence $\mathbf{r}_E(z)=0$ for any $z\not\in\frac{\mu-\lambda}{2}\mathbb{Z}_n$.\;Then from Lemma \ref{l-4.7},\;there exists a subset $\widetilde{E}'$ of $\{1,\cdots,{\frac{2n}{\mu-\lambda}-1}\}$ such that
\begin{equation}
\begin{aligned}
E={\widetilde{E}'}+\frac{2n}{\mu-\lambda}\mathbb{Z}_n\Rightarrow X=x^E=x^{\widetilde{E}'}\langle x^{\frac{2n}{\mu-\lambda}}\rangle\overset{\text{def}}=T'\langle x^{\frac{2n}{\mu-\lambda}}\rangle.
\end{aligned}
\end{equation}
This completes the proof.\;
\end{proof}

\section{Some constructions of directed strongly regular dihedrants}
\begin{con}\label{c-5.1}
Let $v$ be an odd positive divisor of $n$.\;Let $T$ be a subset of $\{x^1,x^2,\cdots,x^{v-1}\}$ and
$X$ be a subset of $C_n$ satisfy the following conditions:\\
$(i)$\;$X=T\langle x^v\rangle$.\\
$(ii)$\;$X\cup X^{(-1)}=C_n\setminus\langle x^v\rangle$.\\
Then $Dih(n,X,X)$ is a DSRG with parameters $\left(2n,n-l,\frac{n-l}{2},\frac{n-l}{2}-l,\frac{n-l}{2}\right)$,\;where $l=\frac{n}{v}$.\;
\end{con}
\begin{proof}Note that $\overline{X}\overline{U_X}=-l\overline{X}+\frac{n-l}{2}\overline{C_n}$.\;The result follows from Lemma \ref{l-4.2} directly.\;
\end{proof}

\begin{con}\label{c-5.2}Let $2v$ be an even positive divisor of $n$.\;Let $T$ be a subset of $\{x^1,x^2,\cdots,x^{2v-1}\}$ and
$X$ be a subset of $C_n$ satisfy the following conditions:\\
$(i)$\;$X=T\langle x^{2v}\rangle$;\\
$(ii)$\;$X\uplus X^{(-1)}=(C_n\setminus\langle x^{2v}\rangle)\uplus x^v\langle x^{2v}\rangle$;\\
$(iii)$\;$X\cup x^vX=C_n$.\\
Then $Dih(n,X,X)$ is a DSRG with parameters $\left(2n,n,\frac{n}{2}+l',\frac{n}{2}-l',\frac{n}{2}+l'\right)$,\;where $l'=\frac{n}{2v}$.\;
\end{con}

\begin{proof}Note that $|X|=\frac{n}{2}$ and $\overline{U_X}=\overline{C_n}-\overline{\langle x^{2v}\rangle}+\overline{x^v\langle x^{2v}\rangle}$.\;Thus $\overline{X}\;\overline{U_X}=-l'\overline{X}+\frac{n}{2}\overline{C_n}+\overline{X}\;\overline{x^v\langle x^{2v}\rangle}=-l'\overline{X}+\frac{n}{2}\overline{C_n}+l'\overline{x^{v}X}
=-l'\overline{X}+\frac{n}{2}\overline{C_n}+l'\overline{C_n}-l'\overline{X}=(\frac{n}{2}+l')\overline{C_n}-2l'\overline{X}$. The result follows from Lemma \ref{l-4.2} directly.\;
\end{proof}
\begin{rmk}In Constructions \ref{c-5.2},\;$v\geqslant2$,\;otherwise,\;$Dih(n,X,X)$ is not a genuine DSRG which has parameters $\left(2n,n,n,0,n\right)$,\;we don't consider this case in this paper.\;
\end{rmk}
\begin{rmk}In Constructions \ref{c-5.1},\;$n\geqslant 3$.\;In Constructions \ref{c-5.2},\;$n\geqslant 4$,\;this is the reason that we assume $n\geqslant 3$ in this paper.\;
\end{rmk}

\section{The characterization of directed strongly regular dihedrants $Dih(n,X,X)$ }
We can now prove the  main theorem  of this paper.\;
\begin{thm}The dihedrant $Dih(n,X,X)$ is a DSRG with parameters $(2n,2|X|,\mu,\lambda,t)$ if and only if one of the following  holds:\\
$(a)$\;Let $v=\frac{n}{\mu-\lambda}$.\;There is a subset $T$ of $\{x,x^2,\cdots,x^{v-1}\}$ satisfies\\
\hspace*{15pt}$(i)$\;$X=T\langle x^v\rangle$;\\
\hspace*{15pt}$(ii)$\;$X\cup X^{(-1)}=C_n\setminus\langle x^v\rangle$.\\
$(b)$\;Let $v=\frac{n}{\mu-\lambda}$.\;There is a subset $T'$ of $\{x,x^2,\cdots,x^{2v-1}\}$ satisfies\\
\hspace*{15pt}$(i)$\;$X=T'\langle x^{2v}\rangle$;\\
\hspace*{15pt}$(ii)$\;$X\uplus X^{(-1)}=(C_n\setminus\langle x^{2v}\rangle)\uplus x^v\langle x^{2v}\rangle$;\\
\hspace*{15pt}$(iii)$\;$X\cup x^vX=C_n$.
\end{thm}
\begin{proof}It follows from Construction \ref{c-5.1} and \ref{c-5.2} that the dihedrant $Dih(n,X,X)$ satisfies condition $(a)$ or $(b)$ is a DSRG.\;Now assume that dihedrant $Dih(n,X,X)$ is a DSRG with parameters $(2n,2|X|,\mu,\lambda,t)$ and we divide into two cases.\;Recall that
$\chi(\overline{U_X})\in\{0,\lambda-\mu\}$ for all nonprincipal characters $\chi$ of $C_n$.\\
$\mathbf{Case\;1}$.\;Suppose $\mu-\lambda$ is an odd number,\;then from Lemma \ref{l-4.9} and (\ref{4.9}),\;we have
\begin{equation}
\begin{aligned}
X=T\langle x^{v}\rangle
\end{aligned}
\end{equation}
for some subset $T$ of $\{x^1,\cdots,x^{v-1}\}$,\;agreeing with the case $(a)$ $(i)$.\;Let $\pi_{v}$ be the natural projection from $C_n=\langle x\rangle$ to the quotient group $\langle x\rangle/\langle x^{v}\rangle$.\;Then
\[\pi_{v}(U_X)=\pi_v({U_T}\langle x^{v}\rangle)=(\mu-\lambda)\oplus \pi_v(U_T).\]
Note that there is a bijection between
\[\widehat{\langle x\rangle/\langle x^{v}\rangle}\leftrightarrow\{\chi\in \widehat{C_n}|\langle x^{v}\rangle\subseteq\ker\chi\}=\{\chi_{(\mu-\lambda)j}|0\leqslant j\leqslant v-1\}.\]
Let $\chi'_j$ be the character of quotient group $\langle x\rangle/\langle x^{v}\rangle$ induced by the character $\chi_{(\mu-\lambda)j}\in \widehat{C_n}$,\;where $1\leqslant j\leqslant v-1$,\;then by Lemma \ref{l-4.6},\;
\begin{equation}\label{6.2}
\begin{aligned}
\chi_j'(\overline{\pi_v(U_T)})=\frac{\chi_j'(\overline{(\mu-\lambda)\oplus  \pi_v(U_T)})}{\mu-\lambda}=\frac{\chi_{(\mu-\lambda)j}(\overline{{U_T}\langle x^{v}\rangle})}{\mu-\lambda}=\frac{\chi_{(\mu-\lambda)j}(\overline{U_X})}{\mu-\lambda}\in\{0,-1\}
\end{aligned}
\end{equation}
for any $1\leqslant j\leqslant v-1$.\;As a direct consequence of Lemma \ref{l-3.8} and (\ref{6.2}),\;we have
\[\pi_v(U_T)=\left(\langle x\rangle/\langle x^{v}\rangle\right)\setminus\{\pi_v(e)\}=\{\pi_v(x^j):1\leqslant j\leqslant v-1\}.\]
This gives that
\[U_T=\{x^jb_j:1\leqslant j\leqslant v-1\}\]
for some $b_1,b_2,\cdots,b_{v-1}\in \langle x^v\rangle$.\;Then
\[X\cup X^{(-1)}=U_X=U_T\langle x^v\rangle =C_n\setminus\langle x^{v}\rangle,\]
agreeing with the case $(a)$ $(ii)$.\;\\
$\mathbf{Case\;2}$.We now assume $\mu-\lambda$ is an even number,\;we can also obtain
\begin{equation}
\begin{aligned}
X={T'}\langle x^{2v}\rangle
\end{aligned}
\end{equation}
for some subset $T'$ of $\{x^1,\cdots,x^{2v-1}\}$ by Lemma \ref{l-4.9}.\;Let $\pi_{2v}$ be the natural projection from $\langle x\rangle$ to $\langle x\rangle/\langle x^{2v}\rangle$.\;Then
$$\pi_{2v}(U_X)=\pi_{2v}({U_{T'}}\langle x^{2v}\rangle)=\frac{\mu-\lambda}{2}\oplus \pi_{2v}(U_{T'}).$$
Let $\chi'_j$ be the character of quotient group $\langle x\rangle/\langle x^{2v}\rangle$ induced by the character $\chi_{\frac{\mu-\lambda}{2}j}\in \widehat{C_n}$,\;where $1\leqslant j\leqslant 2v-1$,\;then by Lemma \ref{l-4.6},\;
\begin{equation}\label{6.4}
\begin{aligned}
\chi_j'(\overline{\pi_{2v}(U_{T'})})=\frac{2\chi_{(\mu-\lambda)j}(\overline{{U_{T'}}\langle x^{2v}\rangle})}{\mu-\lambda}=\frac{2\chi_{(\mu-\lambda)j}(\overline{U_X})}{\mu-\lambda}\in\{0,-2\}
\end{aligned}
\end{equation}
for any $1\leqslant j\leqslant 2v-1$.\;As a direct consequence of Lemma \ref{l-3.9} and (\ref{6.4}),\;we have
$${\pi_{2v}(U_{T'})}=\left(\langle x\rangle/\langle x^{2v}\rangle\right)\setminus\{\pi_{2v}(e),\pi_{2v}(x^{v})\}\;
\text{or}\;\{\pi_{2v}(x^{v})\}\uplus\left(\left(\langle x\rangle/\langle x^{2v}\rangle\right)\setminus\{\pi_{2v}(e)\}\right).$$
$\mathbf{Case\;2.1}$.\;For the first choice of ${\pi_{2v}(U_{T'})}$,\;we can get $|T'|=v-1$ and $U_X=C_n\setminus(\langle x^{2v}\rangle\cup x^v\langle x^{2v}\rangle)=C_n\setminus\langle x^{v}\rangle$,\;agreeing with the case $(a)$ $(ii)$.\;From Lemma \ref{l-4.2},\;we have
\begin{equation*}
\mu\overline{C_n}-\frac{n}{v}\overline{T'\langle x^{2v}\rangle}=\overline{X}\;\overline{U_X}=\overline{T'\langle x^{2v}\rangle}(\overline{C_n}-\overline{\langle x^{v}\rangle})=|X|\overline{C_n}-\frac{n}{2v}\overline{T'\langle x^{v}\rangle}.
\end{equation*}
This implies that
\begin{equation}\label{6.5}
\frac{2v}{n}(|X|-\mu)\overline{C_n}=\overline{T'\langle x^{v}\rangle}-2\overline{T'\langle x^{2v}\rangle}=\overline{T'x^v\langle x^{2v}\rangle}-\overline{T'\langle x^{2v}\rangle}.
\end{equation}
Thus we can get
\begin{equation}\label{6.6}
\frac{2v}{n}(|X|-\mu)n=\chi_0(\overline{T'x^v\langle x^{2v}\rangle})-\chi_0(\overline{T'\langle x^{2v}\rangle})=0\Rightarrow\frac{2v}{n}(|X|-\mu)=0.
\end{equation}
Then (\ref{6.5}) and (\ref{6.6}) gives that
\begin{equation}\label{6.7}
{T'x^{v}\langle x^{2v}\rangle}={T'\langle x^{2v}\rangle}.
\end{equation}
By applying $\pi_{2v}$ on both sides of (\ref{6.7}),\;we have
\[\pi_{2v}(T')\widetilde{x}^v=\pi_{2v}(T')\]
where $\widetilde{x}=\pi_{2v}(x)$.\;Let $\widehat{\langle x\rangle/\langle x^{2v}\rangle}=\{\chi'_j:0\leqslant j\leqslant 2v-1\}$,\;then
\[\chi_j'(\pi_{2v}(T'))=\chi_j'(\pi_{2v}(T'))\chi_j'(\widetilde{x}^v)=-\chi_j'(\pi_{2v}(T'))\Rightarrow \chi_j'(\pi_{2v}(T'))=0\]
for any $0\leqslant j\leqslant 2v-1$ and $j\not\in 2\mathbb{Z}_{2v}$.\;It follows from Corollary \ref{r-4.8} that $\pi_{2v}(T')=x^{E}\langle \widetilde{x}^{v}\rangle$ for some subset $E$ of $\{1,\cdots,{v-1}\}$.\;Therefore
\[T'=\bigcup_{a\in E}\{x^{a}b_a,x^{a+v}b_a'\}\]
for some $b_a,b_a'\in \langle x^{2v}\rangle$,\;where $a\in E$.\;Then
\[X=T'\langle x^{2v}\rangle=x^{E}\langle x^{2v}\rangle\cup x^{E}x^v\langle x^{2v}\rangle= x^{E}\langle x^{v}\rangle\overset{\text{def}}=T\langle x^{v}\rangle,\]
agreeing with the case $(a)$ $(i)$.\\
$\mathbf{Case\;2.2}$.\;For the second choice of ${\pi_{2v}(U_{T'})}$,\;we can get $|T'|=v$,\; $X={T'}\langle x^{2v}\rangle$ and $U_X=(C_n\setminus\langle x^{2v}\rangle)\uplus x^v\langle x^{2v}\rangle$,\;agreeing with the case $(b)$ $(i)$ and $(ii)$.\;From Lemma \ref{l-4.2},\;we have
\begin{equation}
\begin{aligned}
\overline{X}\;\overline{U_X}&=\overline{T'\langle x^{2v}\rangle}(\overline{C_n}-\overline{\langle x^{2v}\rangle})+\overline{T'\langle x^{2v}\rangle}\;\overline{x^v\langle x^{2v}\rangle}=\frac{n}{2}\overline{C_n}-\frac{n}{2v}\overline{T'\langle x^{2v}\rangle}+\frac{n}{2v}\overline{T'x^v\langle x^{2v}\rangle}\\
&=\mu\overline{C_n}-\frac{n}{v}\overline{T'\langle x^{2v}\rangle}.
\end{aligned}
\end{equation}
This claims that
\begin{equation}\label{6.9}
\frac{n}{2v}\overline{T'\langle x^{v}\rangle}=(\mu-\frac{n}{2})\overline{C_n}.
\end{equation}
Hence
$$(\mu-\frac{n}{2})n=\chi_0((\mu-\frac{n}{2})\overline{C_n})=\chi_0(\frac{n}{2v}\overline{T'\langle x^{v}\rangle})=\frac{n}{2v}\cdot|T'|\cdot\frac{n}{2v}=\frac{n}{2}\frac{n}{v},$$
this gives that $(\mu-\frac{n}{2})=\frac{n}{2v}$.\;Then (\ref{6.9}) implies
\[{C_n}={T'\langle x^{v}\rangle}=T'\langle x^{2v}\rangle\cup T'x^v\langle x^{2v}\rangle=X\cup(x^vX),\]
agreeing with the case $(b)$ $(iii)$.\;
\end{proof}
\section*{Acknowledgements}
The author are grateful to those of you who support to us.
\section*{References}

\bibliographystyle{plain}
\bibliography{12}

\end{document}